\documentclass[reqno]{amsart}
\usepackage{amssymb,amsmath}%
\usepackage{ifpdf}
\ifpdf
\usepackage[hyperindex]{hyperref}
\else
\expandafter \ifx \csname dvipdfm \endcsname \relax
\usepackage[hypertex,hyperindex]{hyperref}
\else
\usepackage[dvipdfm,hyperindex]{hyperref}
\usepackage[sort&compress]{cite}
\fi
\fi

\begin{document}
\title[On Tribonacci and Tribonacci-Lucas Quaternion Polynomials]{On Tribonacci and Tribonacci-Lucas Quaternion Polynomials}

\author[G. Cerda-Morales]{Gamaliel Cerda-Morales}

\address{Gamaliel Cerda-Morales \newline
 Instituto de Matem\'aticas, Pontificia Universidad Cat\'olica de Valpara\'iso, Blanco Viel 596, Valpara\'iso, Chile.}
\email{gamaliel.cerda.m@mail.pucv.cl}



\begin{abstract}
In this paper, we introduce the Tribonacci and Tribonacci-Lucas quaternion polynomials. We obtain the Binet formulas, generating functions and exponential generating functions of these quaternions. Moreover, we give some properties and identities for the Tribonacci and Tribonacci-Lucas quaternions.
\end{abstract}

\maketitle
\numberwithin{equation}{section}
\newtheorem{theorem}{Theorem}[section]
\newtheorem{lemma}[theorem]{Lemma}
\newtheorem{proposition}[theorem]{Proposition}
\newtheorem{definition}[theorem]{Definition}
\newtheorem{corollary}[theorem]{Corollary}
\newtheorem*{remark}{Remark}

\section{Introduction}
For any positive real number $x$, the Tribonacci and Tribonacci-Lucas polynomials, $\{T_{n}(x)\}_{n\in \Bbb{N}}$ and $\{t_{n}(x)\}_{n\in \Bbb{N}}$, are defined by, for $n\geq 3$,
\begin{equation}\label{eq:1}
T_{n}(x)=x^{2}T_{n-1}(x)+xT_{n-2}(x)+T_{n-3}(x)
\end{equation}
and 
\begin{equation}\label{eq:2}
t_{n}(x)=x^{2}t_{n-1}(x)+xt_{n-2}(x)+t_{n-3}(x),
\end{equation}
respectively, where $T_{0}(x)=0$, $T_{1}(x)=1$, $T_{2}(x)=x^{2}$, $t_{0}(x)=3$, $t_{1}(x)=x^{2}$ and $t_{2}(x)=x^{4}+2x$.

Let $\alpha(x)$, $\omega_{1}(x)$ and $\omega_{2}(x)$ be the roots of the characteristic equation $\lambda^{3}-x^{2}\lambda^{2}-x\lambda-1=0$. Then, the Binet formulas for the Tribonacci and Tribonacci-Lucas polynomials are given by
\begin{align*}
T_{n}(x)&=\frac{\alpha^{n+1}(x)}{(\alpha(x)-\omega_{1}(x))(\alpha(x)-\omega_{2}(x))}-\frac{\omega^{n+1}(x)}{(\alpha(x)-\omega_{1}(x))(\omega_{1}(x)-\omega_{2}(x))}\\
&\ \ +\frac{\omega_{2}^{n+1}(x)}{(\alpha(x)-\omega_{2}(x))(\omega_{1}(x)-\omega_{2}(x))},\ n\geq0
\end{align*}
and $$t_{n}(x)=\alpha^{n}(x)+\omega_{1}^{n}(x)+\omega_{2}^{n}(x),\ n\geq0$$ with $\alpha(x)=\frac{x^{2}}{3}+A(x)+B(x)$, $\omega_{1}(x)=\frac{x^{2}}{3}+\epsilon A(x)+\epsilon^{2} B(x)$ and $\omega_{2}(x)=\frac{x^{2}}{3}+\epsilon^{2}A(x)+\epsilon B(x)$, where $$A(x)=\sqrt[3]{\frac{x^{6}}{27}+\frac{x^{3}}{6}+\frac{1}{2}+\sqrt{\frac{x^{6}}{37}+\frac{7x^{3}}{54}+\frac{1}{4}}},$$ $$B(x)=\sqrt[3]{\frac{x^{6}}{27}+\frac{x^{3}}{6}+\frac{1}{2}-\sqrt{\frac{x^{6}}{37}+\frac{7x^{3}}{54}+\frac{1}{4}}},$$ with $\epsilon=-\frac{1}{2}+\frac{i\sqrt{3}}{2}$. 

One can easily see that 
\begin{equation}\label{eq:3}
\alpha(x)+\omega_{1}(x)+\omega_{2}(x)=x^{2}\ \textrm{and}\ \alpha(x)\omega_{1}(x)\omega_{2}(x)=1. 
\end{equation}

The generating functions of the Tribonacci and Tribonacci-Lucas polynomials are given by 
\begin{equation}\label{eq:3.1}
G(y)=\sum_{n=0}^{\infty}T_{n}(x)y^{n}=\frac{y}{1-x^{2}y-xy^{2}-y^{3}}
\end{equation}
and
\begin{equation}\label{eq:3.2}
g(y)=\sum_{n=0}^{\infty}t_{n}(x)y^{n}=\frac{3-2x^{2}y-xy^{2}}{1-x^{2}y-xy^{2}-y^{3}},
\end{equation}
where $T_{n}(x)$ is the $n$-th Tribonacci polynomial and $t_{n}(x)$ is the $n$-th Tribonacci-Lucas polynomial. For more details and properties related to the Tribonacci and Tribonacci-Lucas polynomials, we refer to \cite{Ho,Ra1}. Taking $x=1$ in (\ref{eq:3.1}) and (\ref{eq:3.2}), we obtain the generating function of the Tribonacci and Tribonacci-Lucas numbers, respectively, for more of this type of numbers \cite{Al,Mc,Sp}.

A quaternion $q$, with real components $q_{r}$, $q_{i}$, $q_{j}$, $q_{k}$ and basis $\textbf{1}$, $\textbf{i}$, $\textbf{j}$, $\textbf{k}$, is an element of the form $$q=q_{r}+q_{i}\textbf{i}+q_{j}\textbf{j}+q_{k}\textbf{k},\ (q_{r}\textbf{1}=q_{r}),$$ where
\begin{equation}\label{eq:4}
   \begin{gathered}
     \textbf{i}^{2}=\textbf{j}^{2}=\textbf{k}^{2}=\textbf{i}\textbf{j}\textbf{k}=-1,\\
   \textbf{i}\textbf{j}=-\textbf{j}\textbf{i}=\textbf{k},\ \textbf{j}\textbf{k}=-\textbf{k}\textbf{j}=\textbf{i}, \textbf{k}\textbf{i}=-\textbf{i}\textbf{k}=\textbf{j}.
   \end{gathered}
\end{equation}

In \cite{Ho1}, Horadam defined the $n$-th Fibonacci and $n$-th Lucas quaternions as $$Q_{n}=F_{n}+F_{n+1}\textbf{i}+F_{n+2}\textbf{j}+F_{n+3}\textbf{k},\ n\geq0$$ and $$K_{n}=L_{n}+L_{n+1}\textbf{i}+L_{n+2}\textbf{j}+L_{n+3}\textbf{k},\ n\geq0$$
respectively, where $F_{n}$ and $L_{n}$ are the $n$-th Fibonacci number and the $n$-th Lucas number and $\textbf{i}$, $\textbf{j}$, $\textbf{k}$ satisfy the multiplication rules (\ref{eq:4}).

Recently, in \cite{Ce2}, we defined the $n$-th generalized Tribonacci quaternion as $$Q_{v,n}=V_{n}+V_{n+1}\textbf{i}+V_{n+2}\textbf{j}+V_{n+3}\textbf{k},\ n\geq0,$$ where $V_{n}$ is the $n$-th generalized Tribonacci number and $\textbf{i}$, $\textbf{j}$, $\textbf{k}$ satisfy the multiplication rules (\ref{eq:4}). 

The Fibonacci and Lucas quaternions have been studied in several papers. For instance, Iyer \cite{Iy1,Iy2} gave some relations connecting the Fibonacci and Lucas quaternions. Iakin \cite{Ia1,Ia3} introduced the concept of a higher order quaternion and generalized quaternions with quaternion components. In \cite{Ho2}, Horadam studied the quaternion recurrence relations. Swamy \cite{Sw} derived the relations of generalized Fibonacci quaternions. Halici \cite{Ha} derived the generating functions and many other identities for the Fibonacci and Lucas quaternions. Akyi\u{g}it et al. \cite{Ak1,Ak2} introduced the Fibonacci generalized quaternions and split Fibonacci quaternions. Cerda-Morales \cite{Ce1} gave some properties of third order Jacobsthal quaternions in a generalized quaternion algebra. Ram\'irez \cite{Ra2} has obtained some combinatorial properties of the $k$-Fibonacci and the $k$-Lucas quaternions. Catarino \cite{Ca} has derived some properties of the $h(x)$-Fibonacci quaternion polynomials.  Szynal-Liana and W\l och \cite{Sz} have introduced Pell quaternions and Pell octonions. Bolat and \.{I}pek \cite{Bo} have given various identities related to Pell quaternions and Pell-Lucas quaternions.

Inspired by these, in this paper, we introduce the Tribonacci and Tribonacci-Lucas quaternion polynomials. We obtain the Binet formulas, generating functions and exponential generating functions of these quaternions. Moreover, we give some properties for the Tribonacci and Tribonacci-Lucas quaternions.

\section{Some properties of the Tribonacci and Tribonacci-Lucas quaternion polynomials}

\begin{definition}\label{Def:1}
For $n\geq 0$, the Tribonacci and Tribonacci-Lucas quaternion polynomials are defined by
\begin{equation}\label{eq:5}
Q_{T,n}(x)=T_{n}(x)+T_{n+1}(x)\textbf{i}+T_{n+2}(x)\textbf{j}+T_{n+3}(x)\textbf{k}
\end{equation}
and
\begin{equation}\label{eq:6}
Q_{t,n}(x)=t_{n}(x)+t_{n+1}(x)\textbf{i}+t_{n+2}(x)\textbf{j}+t_{n+3}(x)\textbf{k}
\end{equation}
where $T_{n}(x)$ and $t_{n}(x)$ are the $n$-th Tribonacci polynomial and the $n$-th Tribonacci-Lucas polynomial. Here $\textbf{i}$, $\textbf{j}$, $\textbf{k}$ are quaternionic units which satisfy the multiplication rules (\ref{eq:4}).
\end{definition}

\begin{proposition}\label{prop:1}
For $n\geq0$, the following identities hold:
\begin{enumerate}
\item[(i).] $Q_{T,n+3}(x)=x^{2}Q_{T,n+2}(x)+xQ_{T,n+1}(x)+Q_{T,n}(x)$.
\item[(ii).] $Q_{t,n+3}(x)=x^{2}Q_{t,n+2}(x)+xQ_{t,n+1}(x)+Q_{t,n}(x)$.
\end{enumerate}
\end{proposition}
\begin{proof}
(i) From equations (\ref{eq:1}) and (\ref{eq:5}), we obtain 
\begin{align*}
x^{2}&Q_{T,n+2}(x)+xQ_{T,n+1}(x)+Q_{T,n}(x)\\
&=x^{2}(T_{n+2}(x)+T_{n+3}(x)\textbf{i}+T_{n+4}(x)\textbf{j}+T_{n+5}(x)\textbf{k})\\
&\ \ +x(T_{n+1}(x)+T_{n+2}(x)\textbf{i}+T_{n+3}(x)\textbf{j}+T_{n+4}(x)\textbf{k})\\
&\ \ +T_{n}(x)+T_{n+1}(x)\textbf{i}+T_{n+2}(x)\textbf{j}+T_{n+3}(x)\textbf{k}\\
&=(x^{2}T_{n+2}(x)+xT_{n+1}(x)+T_{n}(x))+(x^{2}T_{n+3}(x)+xT_{n+2}(x)+T_{n+1}(x))\textbf{i}\\
&\ \ +(x^{2}T_{n+4}(x)+xT_{n+3}(x)+T_{n+2}(x))\textbf{j}+(x^{2}T_{n+5}(x)+xT_{n+4}(x)+T_{n+3}(x))\textbf{k}\\
&=T_{n+3}(x)+T_{n+4}(x)\textbf{i}+T_{n+5}(x)\textbf{j}+T_{n+6}(x)\textbf{k}\\
&=Q_{T,n+3}(x).
\end{align*}
(ii) The proof is similar to (i), using the equations (\ref{eq:2}) and (\ref{eq:6}).
\end{proof}

\begin{theorem}[Binet formulas]\label{thm:1}
For $n\geq0$, we have
\begin{equation}\label{eq:7}
T_{n}(x)=\left(
\begin{array}{c}
\frac{\underline{\alpha}\alpha^{n+1}(x)}{(\alpha(x)-\omega_{1}(x))(\alpha(x)-\omega_{2}(x))}-\frac{\underline{\omega_{1}}\omega_{1}^{n+1}(x)}{(\alpha(x)-\omega_{1}(x))(\omega_{1}(x)-\omega_{2}(x))} \\
+\frac{\underline{\omega_{2}}\omega_{2}^{n+1}(x)}{(\alpha(x)-\omega_{2}(x))(\omega_{1}(x)-\omega_{2}(x))}
\end{array}%
\right)
\end{equation}
and
\begin{equation}\label{eq:8}
t_{n}(x)=\underline{\alpha}\alpha^{n}(x)+\underline{\omega_{1}}\omega_{1}^{n}(x)+\underline{\omega_{2}}\omega_{2}^{n}(x),
\end{equation}
where $\underline{\alpha}=1+\alpha(x)\textbf{i}+\alpha^{2}(x)\textbf{j}+\alpha^{3}(x)\textbf{k}$, $\underline{\omega_{1}}=1+\omega_{1}(x)\textbf{i}+\omega_{1}^{2}(x)\textbf{j}+\omega_{1}^{3}(x)\textbf{k}$, $\underline{\omega_{2}}=1+\omega_{2}(x)\textbf{i}+\omega_{2}^{2}(x)\textbf{j}+\omega_{2}^{3}(x)\textbf{k}$ and $\textbf{i}$, $\textbf{j}$, $\textbf{k}$ are quaternion units which satisfy the multiplication rules (\ref{eq:4}).
\end{theorem}
\begin{proof}
Using the Definition \ref{Def:1} and the Binet formulas for the Tribonacci and Tribonacci-Lucas polynomials, we obtain
\begin{align*}
Q_{T,n}(x)&=T_{n}(x)+T_{n+1}(x)\textbf{i}+T_{n+2}(x)\textbf{j}+T_{n+3}(x)\textbf{k}\\
&=\left(
\begin{array}{c}
\frac{\alpha^{n+1}(x)}{(\alpha(x)-\omega_{1}(x))(\alpha(x)-\omega_{2}(x))}-\frac{\omega_{1}^{n+1}(x)}{(\alpha(x)-\omega_{1}(x))(\omega_{1}(x)-\omega_{2}(x))} \\
+\frac{\omega_{2}^{n+1}(x)}{(\alpha(x)-\omega_{2}(x))(\omega_{1}(x)-\omega_{2}(x))}
\end{array}%
\right)\\
&\ \ +\left(
\begin{array}{c}
\frac{\alpha^{n+2}(x)}{(\alpha(x)-\omega_{1}(x))(\alpha(x)-\omega_{2}(x))}-\frac{\omega_{1}^{n+2}(x)}{(\alpha(x)-\omega_{1}(x))(\omega_{1}(x)-\omega_{2}(x))} \\
+\frac{\omega_{2}^{n+2}(x)}{(\alpha(x)-\omega_{2}(x))(\omega_{1}(x)-\omega_{2}(x))}
\end{array}%
\right)\textbf{i}\\
&\ \ +\left(
\begin{array}{c}
\frac{\alpha^{n+3}(x)}{(\alpha(x)-\omega_{1}(x))(\alpha(x)-\omega_{2}(x))}-\frac{\omega_{1}^{n+3}(x)}{(\alpha(x)-\omega_{1}(x))(\omega_{1}(x)-\omega_{2}(x))} \\
+\frac{\omega_{2}^{n+3}(x)}{(\alpha(x)-\omega_{2}(x))(\omega_{1}(x)-\omega_{2}(x))}
\end{array}%
\right)\textbf{j}\\
&\ \ +\left(
\begin{array}{c}
\frac{\alpha^{n+4}(x)}{(\alpha(x)-\omega_{1}(x))(\alpha(x)-\omega_{2}(x))}-\frac{\omega_{1}^{n+4}(x)}{(\alpha(x)-\omega_{1}(x))(\omega_{1}(x)-\omega_{2}(x))} \\
+\frac{\omega_{2}^{n+4}(x)}{(\alpha(x)-\omega_{2}(x))(\omega_{1}(x)-\omega_{2}(x))}
\end{array}%
\right)\textbf{k}\\
&=\left(
\begin{array}{c}
\frac{\alpha^{n+1}(x)}{(\alpha(x)-\omega_{1}(x))(\alpha(x)-\omega_{2}(x))}(1+\alpha(x)\textbf{i}+\alpha^{2}(x)\textbf{j}+\alpha^{3}(x)\textbf{k})\\
-\frac{\omega_{1}^{n+1}(x)}{(\alpha(x)-\omega_{1}(x))(\omega_{1}(x)-\omega_{2}(x))}(1+\omega_{1}(x)\textbf{i}+\omega_{1}^{2}(x)\textbf{j}+\omega_{1}^{3}(x)\textbf{k}) \\
+\frac{\omega_{2}^{n+1}(x)}{(\alpha(x)-\omega_{2}(x))(\omega_{1}(x)-\omega_{2}(x))}(1+\omega_{2}(x)\textbf{i}+\omega_{2}^{2}(x)\textbf{j}+\omega_{2}^{3}(x)\textbf{k})
\end{array}%
\right)\\
&=\left(
\begin{array}{c}
\frac{\underline{\alpha}\alpha^{n+1}(x)}{(\alpha(x)-\omega_{1}(x))(\alpha(x)-\omega_{2}(x))}-\frac{\underline{\omega_{1}}\omega_{1}^{n+1}(x)}{(\alpha(x)-\omega_{1}(x))(\omega_{1}(x)-\omega_{2}(x))} \\
+\frac{\underline{\omega_{2}}\omega_{2}^{n+1}(x)}{(\alpha(x)-\omega_{2}(x))(\omega_{1}(x)-\omega_{2}(x))}
\end{array}%
\right)
\end{align*}
and
\begin{align*}
Q_{t,n}(x)&=t_{n}(x)+t_{n+1}(x)\textbf{i}+t_{n+2}(x)\textbf{j}+t_{n+3}(x)\textbf{k}\\
&=\alpha^{n}(x)+\omega_{1}^{n}(x)+\omega_{2}^{n}(x)+(\alpha^{n+1}(x)+\omega_{1}^{n+1}(x)+\omega_{2}^{n+1}(x))\textbf{i}\\
&\ \ +(\alpha^{n+2}(x)+\omega_{1}^{n+2}(x)+\omega_{2}^{n+2}(x))\textbf{j}+(\alpha^{n+3}(x)+\omega_{1}^{n+3}(x)+\omega_{2}^{n+3}(x))\textbf{k}\\
&=\alpha^{n}(x)(1+\alpha(x)\textbf{i}+\alpha^{2}(x)\textbf{j}+\alpha^{3}(x)\textbf{k})+\omega_{1}^{n}(x)(1+\omega_{1}(x)\textbf{i}+\omega_{1}^{2}(x)\textbf{j}+\omega_{1}^{3}(x)\textbf{k})\\
&\ \ +\omega_{2}^{n}(x)(1+\omega_{2}(x)\textbf{i}+\omega_{2}^{2}(x)\textbf{j}+\omega_{2}^{3}(x)\textbf{k})\\
&=\underline{\alpha}\alpha^{n}(x)+\underline{\omega_{1}}\omega_{1}^{n}(x)+\underline{\omega_{2}}\omega_{2}^{n}(x).
\end{align*}
\end{proof}

\begin{theorem}\label{thm:2}
The generating functions for the Tribonacci and Tribonacci-Lucas quaternion polynomials are
$$G_{T}(y)=\frac{y+\textbf{i}+(x^{2}+xy+y^{2})\textbf{j}+(x^{4}+x+x^{3}y+y+x^{2}y^{2})\textbf{k}}{1-x^{2}y-xy^{2}-y^{3}}$$ and $$g_{t}(y)=\frac{\left(
\begin{array}{c}
3-2x^{2}y-xy^{2}+(x^{2}+2xy+3y^{2})\textbf{i}+(x^{4}+2x+x^{3}y+3y+x^{2}y^{2})\textbf{j}\\
+(x^{6}+3x^{3}+3+x^{5}y+3x^{2}y+x^{4}y^{2}+2xy^{2})\textbf{k}
\end{array}%
\right)}{1-x^{2}y-xy^{2}-y^{3}},$$ respectively.
\end{theorem}
\begin{proof}
Let $G_{T}(y)=\sum_{n=0}^{\infty}Q_{T,n}(x)y^{n}$ and $g_{t}(y)=\sum_{n=0}^{\infty}Q_{t,n}(x)y^{n}$. Then we get
the following equation
\begin{align*}
(1-&x^{2}y-xy^{2}-y^{3})G_{T}(y)\\
&=Q_{T,0}(x)+(Q_{T,1}(x)-x^{2}Q_{T,0}(x))y+(Q_{T,2}(x)-x^{2}Q_{T,1}(x)-xQ_{T,0}(x))y^{2}\\
&\ \ +\sum_{n=3}^{\infty}(Q_{T,n}(x)-x^{2}Q_{T,n-1}(x)-xQ_{T,n-2}(x)-Q_{T,n-3}(x))y^{n}.
\end{align*}

Since, for each $n\geq3$, the coefficient of $y^{n}$ is zero in the right-hand side of this equation, we obtain 
\begin{align*}
G_{T}(y)&=\frac{Q_{T,0}(x)+(Q_{T,1}(x)-x^{2}Q_{T,0}(x))y+(Q_{T,2}(x)-x^{2}Q_{T,1}(x)-xQ_{T,0}(x))y^{2}}{1-x^{2}y-xy^{2}-y^{3}}\\
&=\frac{y+\textbf{i}+(x^{2}+xy+y^{2})\textbf{j}+(x^{4}+x+x^{3}y+y+x^{2}y^{2})\textbf{k}}{1-x^{2}y-xy^{2}-y^{3}}.
\end{align*}
Similarly, we get 
\begin{align*}
g_{t}(y)&=\frac{Q_{t,0}(x)+(Q_{t,1}(x)-x^{2}Q_{t,0}(x))y+(Q_{t,2}(x)-x^{2}Q_{t,1}(x)-xQ_{t,0}(x))y^{2}}{1-x^{2}y-xy^{2}-y^{3}}\\
&=\frac{\left(
\begin{array}{c}
3-2x^{2}y-xy^{2}+(x^{2}+2xy+3y^{2})\textbf{i}+(x^{4}+2x+x^{3}y+3y+x^{2}y^{2})\textbf{j}\\
+(x^{6}+3x^{3}+3+x^{5}y+3x^{2}y+x^{4}y^{2}+2xy^{2})\textbf{k}
\end{array}%
\right)}{1-x^{2}y-xy^{2}-y^{3}}.
\end{align*}
\end{proof}

\begin{theorem}\label{thm:3}
For $m\geq 2$, the generating functions for the Tribonacci quaternion polynomial $\{Q_{T,n+m}(x)\}_{n\geq0}$ and the Tribonaci-Lucas quaternion polynomial $\{Q_{t,n+m}(x)\}_{n\geq0}$ are $$\sum_{n=0}^{\infty}Q_{T,n+m}(x)y^{n}=\frac{Q_{T,m}(x)+(xQ_{T,m-1}(x)+Q_{T,m-2})y+Q_{T,m-1}(x)y^{2}}{1-x^{2}y-xy^{2}-y^{3}}$$ and $$\sum_{n=0}^{\infty}Q_{t,n+m}(x)y^{n}=\frac{Q_{t,m}(x)+(xQ_{t,m-1}(x)+Q_{t,m-2})y+Q_{t,m-1}(x)y^{2}}{1-x^{2}y-xy^{2}-y^{3}}.$$
\end{theorem}
\begin{proof}
Using Theorem \ref{thm:1} and Theorem \ref{thm:2}, we obtain 
\begin{align*}
\sum_{n=0}^{\infty}&Q_{t,n+m}(x)y^{n}\\
&=\sum_{n=0}^{\infty}(\underline{\alpha}\alpha^{n+m}(x)+\underline{\omega_{1}}\omega_{1}^{n+m}(x)+\underline{\omega_{2}}\omega_{2}^{n+m}(x))y^{n}\\
&=\underline{\alpha}\alpha^{m}(x)\sum_{n=0}^{\infty}(\alpha(x)y)^{n}+\underline{\omega_{1}}\omega_{1}^{m}(x)\sum_{n=0}^{\infty}(\omega_{1}(x)y)^{n}+\underline{\omega_{2}}\omega_{2}^{m}(x)\sum_{n=0}^{\infty}(\omega_{2}(x)y)^{n}\\
&=\underline{\alpha}\alpha^{m}(x)\frac{1}{1-\alpha(x)y}+\underline{\omega_{1}}\omega_{1}^{m}(x)\frac{1}{1-\omega_{1}(x)y}+\underline{\omega_{2}}\omega_{2}^{m}(x)\frac{1}{1-\omega_{2}(x)y}\\
&=\frac{Q_{t,m}(x)+(xQ_{t,m-1}(x)+Q_{t,m-2})y+Q_{t,m-1}(x)y^{2}}{1-x^{2}y-xy^{2}-y^{3}}.
\end{align*}
\end{proof}

\begin{theorem}\label{thm:4}
For $n\in \Bbb{N}$, the exponential generating functions for the Tribonacci and Tribonacci-Lucas quaternion polynomials are 
\begin{align*}
\sum_{n=0}^{\infty}\frac{Q_{T,n}}{n!}y^{n}&=\frac{\underline{\alpha}\alpha(x)e^{\alpha(x)y}}{(\alpha(x)-\omega_{1}(x))(\alpha(x)-\omega_{2}(x))}-\frac{\underline{\omega_{1}}\omega_{1}(x)e^{\omega_{1}(x)y}}{(\alpha(x)-\omega_{1}(x))(\omega_{1}(x)-\omega_{2}(x))}\\
&\ \ +\frac{\underline{\omega_{2}}\omega_{2}(x)e^{\omega_{2}(x)y}}{(\alpha(x)-\omega_{2}(x))(\omega_{1}(x)-\omega_{2}(x))}
\end{align*}
 and $$\sum_{n=0}^{\infty}\frac{Q_{t,n}}{n!}y^{n}=\underline{\alpha}e^{\alpha(x)y}+\underline{\omega_{1}}e^{\omega_{1}(x)y}+\underline{\omega_{2}}e^{\omega_{2}(x)y},$$ respectively, where $\underline{\alpha}=1+\alpha(x)\textbf{i}+\alpha^{2}(x)\textbf{j}+\alpha^{3}(x)\textbf{k}$, $\underline{\omega_{1}}=1+\omega_{1}(x)\textbf{i}+\omega_{1}^{2}(x)\textbf{j}+\omega_{1}^{3}(x)\textbf{k}$, $\underline{\omega_{2}}=1+\omega_{2}(x)\textbf{i}+\omega_{2}^{2}(x)\textbf{j}+\omega_{2}^{3}(x)\textbf{k}$ and $\textbf{i}$, $\textbf{j}$, $\textbf{k}$ are quaternion units which satisfy the multiplication rules (\ref{eq:4}).
\end{theorem}
\begin{proof}
By considering the Binet formulas for the Tribonacci and Tribonacci-Lucas quaternion polynomials given in Theorem \ref{thm:1}, we get
\begin{align*}
\sum_{n=0}^{\infty}\frac{Q_{T,n}}{n!}y^{n}&=\sum_{n=0}^{\infty}\left(
\begin{array}{c}
\frac{\underline{\alpha}\alpha^{n+1}(x)}{(\alpha(x)-\omega_{1}(x))(\alpha(x)-\omega_{2}(x))}-\frac{\underline{\omega_{1}}\omega_{1}^{n+1}(x)}{(\alpha(x)-\omega_{1}(x))(\omega_{1}(x)-\omega_{2}(x))} \\
+\frac{\underline{\omega_{2}}\omega_{2}^{n+1}(x)}{(\alpha(x)-\omega_{2}(x))(\omega_{1}(x)-\omega_{2}(x))}
\end{array}%
\right)\frac{y^{n}}{n!}\\
&=\frac{\underline{\alpha}\alpha(x)}{(\alpha(x)-\omega_{1}(x))(\alpha(x)-\omega_{2}(x))}\sum_{n=0}^{\infty}\frac{(\alpha(x)y)^{n}}{n!}\\
&\ \ -\frac{\underline{\omega_{1}}\omega_{1}(x)}{(\alpha(x)-\omega_{1}(x))(\omega_{1}(x)-\omega_{2}(x))}\sum_{n=0}^{\infty}\frac{(\omega_{1}(x)y)^{n}}{n!}\\
&\ \ +\frac{\underline{\omega_{2}}\omega_{2}(x)}{(\alpha(x)-\omega_{2}(x))(\omega_{1}(x)-\omega_{2}(x))}\sum_{n=0}^{\infty}\frac{(\omega_{2}(x)y)^{n}}{n!}\\
&=\frac{\underline{\alpha}\alpha(x)e^{\alpha(x)y}}{(\alpha(x)-\omega_{1}(x))(\alpha(x)-\omega_{2}(x))}-\frac{\underline{\omega_{1}}\omega_{1}(x)e^{\omega_{1}(x)y}}{(\alpha(x)-\omega_{1}(x))(\omega_{1}(x)-\omega_{2}(x))}\\
&\ \ +\frac{\underline{\omega_{2}}\omega_{2}(x)e^{\omega_{2}(x)y}}{(\alpha(x)-\omega_{2}(x))(\omega_{1}(x)-\omega_{2}(x))}
\end{align*}
and 
\begin{align*}
\sum_{n=0}^{\infty}\frac{Q_{t,n}}{n!}y^{n}&=\sum_{n=0}^{\infty}(\underline{\alpha}\alpha^{n}(x)+\underline{\omega_{1}}\omega_{1}^{n}(x)+\underline{\omega_{2}}\omega_{2}^{n}(x))\frac{y^{n}}{n!}\\
&=\underline{\alpha}\sum_{n=0}^{\infty}\frac{(\alpha(x)y)^{n}}{n!}+\underline{\omega_{1}}\sum_{n=0}^{\infty}\frac{(\omega_{1}y)^{n}}{n!}+\underline{\omega_{2}}\sum_{n=0}^{\infty}\frac{(\omega_{2}y)^{n}}{n!}\\
&=\underline{\alpha}e^{\alpha(x)y}+\underline{\omega_{1}}e^{\omega_{1}(x)y}+\underline{\omega_{2}}e^{\omega_{2}(x)y}.
\end{align*}
\end{proof}
\begin{theorem}
For $n\geq0$ and related with Tribonacci and Tribonacci-Lucas quaternion polynomials, we have $$\sum_{r=0}^{n}\sum_{s=0}^{r}\binom{n}{r} \binom{r}{s} x^{r+s}Q_{T,r+s}(x)=Q_{T,3n}(x)$$ and $$\sum_{r=0}^{n}\sum_{s=0}^{r}\binom{n}{r} \binom{r}{s} x^{r+s}Q_{t,r+s}(x)=Q_{t,3n}(x),$$ respectively.
\end{theorem}
\begin{proof}
By the Binet formulas for the Tribonacci and Tribonacci-Lucas quaternion polynomials, we have
\begin{align*}
\sum_{r=0}^{n}\sum_{s=0}^{r}&\binom{n}{r} \binom{r}{s} x^{r+s}Q_{T,r+s}(x)\\
&=\sum_{r=0}^{n}\sum_{s=0}^{r}\binom{n}{r} \binom{r}{s} x^{r+s}\left(\frac{\underline{\alpha}\alpha^{r+s+1}(x)}{(\alpha(x)-\omega_{1}(x))(\alpha(x)-\omega_{2}(x))}
\right)\\
&\ \ - \sum_{r=0}^{n}\sum_{s=0}^{r}\binom{n}{r} \binom{r}{s} x^{r+s}\left(\frac{\underline{\omega_{1}}\omega_{1}^{r+s+1}(x)}{(\alpha(x)-\omega_{1}(x))(\omega_{1}(x)-\omega_{2}(x))}\right)\\
&\ \ + \sum_{r=0}^{n}\sum_{s=0}^{r}\binom{n}{r} \binom{r}{s} x^{r+s}\left(\frac{\underline{\omega_{2}}\omega_{2}^{r+s+1}(x)}{(\alpha(x)-\omega_{2}(x))(\omega_{1}(x)-\omega_{2}(x))}
\right)\\
&=\frac{\underline{\alpha}\alpha(x)}{(\alpha(x)-\omega_{1}(x))(\alpha(x)-\omega_{2}(x))}\sum_{r=0}^{n}\binom{n}{r}(x\alpha(x)+x^{2}\alpha^{2}(x))^{r}\\
&\ \ - \frac{\underline{\omega_{1}}\omega_{1}(x)}{(\alpha(x)-\omega_{1}(x))(\omega_{1}(x)-\omega_{2}(x))}\sum_{r=0}^{n}\binom{n}{r}(x\omega_{1}(x)+x^{2}\omega_{1}^{2}(x))^{r}\\
&\ \ +\frac{\underline{\omega_{2}}\omega_{2}(x)}{(\alpha(x)-\omega_{2}(x))(\omega_{1}(x)-\omega_{2}(x))}\sum_{r=0}^{n}\binom{n}{r}(x\omega_{2}(x)+x^{2}\omega_{2}^{2}(x))^{r}\\
&=\frac{\underline{\alpha}\alpha^{3n+1}(x)}{(\alpha(x)-\omega_{1}(x))(\alpha(x)-\omega_{2}(x))}- \frac{\underline{\omega_{1}}\omega_{1}^{3n+1}(x)}{(\alpha(x)-\omega_{1}(x))(\omega_{1}(x)-\omega_{2}(x))}\\
&\ \ +\frac{\underline{\omega_{2}}\omega_{2}^{3n+1}(x)}{(\alpha(x)-\omega_{2}(x))(\omega_{1}(x)-\omega_{2}(x))}\\
&=Q_{T,3n}(x)
\end{align*}
and
\begin{align*}
\sum_{r=0}^{n}\sum_{s=0}^{r}&\binom{n}{r} \binom{r}{s} x^{r+s}Q_{t,r+s}(x)\\
&=\sum_{r=0}^{n}\sum_{s=0}^{r}\binom{n}{r} \binom{r}{s} x^{r+s}\left(\underline{\alpha}\alpha^{r+s}(x)+\underline{\omega_{1}}\omega_{1}^{r+s}(x)+ \underline{\omega_{2}}\omega_{2}^{r+s}(x)\right)\\
&=\underline{\alpha}\sum_{r=0}^{n}\binom{n}{r}(x\alpha(x)+x^{2}\alpha^{2}(x))^{r}+\underline{\omega_{1}}\sum_{r=0}^{n}\binom{n}{r}(x\omega_{1}(x)+x^{2}\omega_{1}^{2}(x))^{r}\\
&\ \ + \underline{\omega_{2}}\sum_{r=0}^{n}\binom{n}{r}(x\omega_{2}(x)+x^{2}\omega_{2}^{2}(x))^{r}\\
&=\underline{\alpha} \alpha^{3n}(x)+\underline{\omega_{1}}\omega_{1}^{3n}(x)+\underline{\omega_{2}}\omega_{2}^{3n}(x)\\
&=Q_{t,3n}(x).
\end{align*}
\end{proof}

\begin{theorem}
For $x\in \Bbb{R}\setminus \{-1,0\}$. The summation formula for Tribonacci quaternion polynomial is as follows:
\begin{equation}\label{eq:13}
\sum_{l=0}^{n}Q_{T,l}(x)=\frac{1}{\delta(x)}\left(Q_{T,n+2}(x)+(1-x^{2})Q_{T,n+1}(x)+Q_{T,n}(x)-\omega(x)\right),
\end{equation}
where $\omega(x)=1+\textbf{i}+(x^{2}+x+1)\textbf{j}+(x^{4}+x^{3}+x^{2}+x+1)\textbf{k}$ and $\delta(x)=x^{2}+x$.
\end{theorem}
\begin{proof}
Using Eq. (\ref{eq:5}), we have
\begin{align*}
\sum_{l=0}^{n}Q_{T,l}(x)&=\sum_{l=0}^{n}T_{l}(x)+\textbf{i}\sum_{l=0}^{n}T_{l+1}(x)+\textbf{j}\sum_{l=0}^{n}T_{l+2}(x)+\textbf{k}\sum_{l=0}^{n}T_{l+3}(x)\\
&=(T_{0}(x)+T_{1}(x)+\cdots+T_{n}(x))+\textbf{i}(T_{1}(x)+T_{2}(x)+\cdots+T_{n+1}(x))\\
&\ \ +\textbf{j}(T_{2}(x)+T_{3}(x)+\cdots+T_{n+2}(x))+\textbf{k}(T_{3}(x)+T_{4}(x)+\cdots+T_{n+3}(x)).
\end{align*}
Since from $\sum_{l=0}^{n}T_{l}(x)=\frac{1}{\delta(x)}(T_{n+2}(x)+(1-x^{2})T_{n+1}(x)+T_{n}(x)-1)$ and using the notation $\delta(x)=x^{2}+x$, we can write 
\begin{align*}
\delta(x)\sum_{l=0}^{n}Q_{T,l}(x)&=T_{n+2}(x)+(1-x^{2})T_{n+1}(x)+T_{n}(x)-1\\
&\ \ +\textbf{i}\left(T_{n+3}(x)+(1-x^{2})T_{n+2}(x)+T_{n+1}(x)-1\right)\\
&\ \ +\textbf{j}\left(T_{n+4}(x)+(1-x^{2})T_{n+3}(x)+T_{n+2}(x)-(1+\delta(x))\right)\\
&\ \ +\textbf{k}\left(T_{n+5}(x)+(1-x^{2})T_{n+4}(x)+T_{n+3}(x)-(1+\delta(x)(1+x^{2}))\right)\\
&=Q_{T,n+2}(x)+(1-x^{2})Q_{T,n+1}(x)+Q_{T,n}(x)-\omega(x),
\end{align*}
where $\omega(x)=1+\textbf{i}+(x^{2}+x+1)\textbf{j}+(x^{4}+x^{3}+x^{2}+x+1)\textbf{k}$. Finally, $$\sum_{l=0}^{n}Q_{T,l}(x)=\frac{1}{\delta(x)}\left(Q_{T,n+2}(x)+(1-x^{2})Q_{T,n+1}(x)+Q_{T,n}(x)+\omega(x)\right).$$
The theorem is proved.
\end{proof}

\section{Matrix Representation of Tribonacci Quaternion Polynomials}
The most useful technique for generating $\{Q_{T,n}(x)\}$ is by means of what we call the $S(x)$-matrix which has been defined and used in \cite{Sh} and is a generalization of the $R$-matrix defined in \cite{Wa}. We defined the $S(x)$-matrix by
\begin{equation}\label{eq:10}
\left[ 
\begin{array}{c}
T_{n+2}(x) \\ 
T_{n+1}(x) \\ 
T_{n}(x)%
\end{array}%
\right]=\left[ 
\begin{array}{ccc}
x^{2} & x & 1 \\ 
1 & 0 & 0 \\ 
0 & 1 & 0
\end{array}%
\right]^{n}\left[ 
\begin{array}{c}
T_{2}(x) \\ 
T_{1}(x) \\ 
T_{0}(x)%
\end{array}%
\right]
\end{equation}
and
\begin{equation}\label{eq:11}
S^{n}(x)=\left[ 
\begin{array}{ccc}
x^{2} & x & 1 \\ 
1 & 0 & 0 \\ 
0 & 1 & 0
\end{array}%
\right]^{n}=\left[ 
\begin{array}{ccc}
T_{n+1}(x) & xT_{n}(x)+T_{n-1}(x) & T_{n}(x) \\ 
T_{n}(x) & xT_{n-1}(x)+T_{n-2}(x) & T_{n-1}(x) \\ 
T_{n-1}(x) & xT_{n-2}(x)+T_{n-3}(x) & T_{n-2}(x)
\end{array}%
\right],
\end{equation}
where $T_{-1}(x)=0$, $T_{-2}(x)=1$ and $T_{-3}(x)=-x$ for convenience.

Now, let us define the following matrix as
\begin{equation}\label{eq:12}
Q_{S}(x)=\left[ 
\begin{array}{ccc}
Q_{T,4}(x) & xQ_{T,3}(x)+Q_{T,2}(x) & Q_{T,3}(x) \\ 
Q_{T,3}(x) & xQ_{T,2}(x)+Q_{T,1}(x) & Q_{T,2}(x) \\ 
Q_{T,2}(x) & xQ_{T,1}(x)+Q_{T,0}(x) & Q_{T,1}(x)
\end{array}
\right].
\end{equation}
This matrix can be called as the Tribonacci quaternion polynomial matrix. Then, we can give the next theorem to the $Q_{S}(x)$-matrix.

\begin{theorem}
If $Q_{T,n}(x)$ be the $n$-th Tribonacci quaternion polynomial. Then, for $n\geq0$:
\begin{equation}\label{eq:13}
Q_{S}(x)\cdot\left[ 
\begin{array}{ccc}
x^{2} & x & 1 \\ 
1 & 0 & 0 \\ 
0 & 1 & 0
\end{array}%
\right]^{n}=\left[ 
\begin{array}{ccc}
Q_{T,n+4}(x) & P_{T,n+3}(x) & Q_{T,n+3}(x) \\ 
Q_{T,n+3}(x) & P_{T,n+2}(x) & Q_{T,n+2}(x) \\ 
Q_{T,n+2}(x) & P_{T,n+1}(x) & Q_{T,n+1}(x)%
\end{array}%
\right],
\end{equation}
where $P_{T,n}(x)=xQ_{T,n}(x)+Q_{T,n-1}(x)$.
\end{theorem}
\begin{proof}
(By induction on $n$) If $n=0$, then the result is obvious. Now, we suppose it is true for $n=m$, that is
$$Q_{S(x)}\cdot S^{m}(x)=\left[ 
\begin{array}{ccc}
Q_{T,m+4}(x) & P_{T,m+3}(x) & Q_{T,m+3}(x) \\ 
Q_{T,m+3}(x) & P_{T,m+2}(x) & Q_{T,m+2}(x) \\ 
Q_{T,m+2}(x) & P_{T,m+1}(x) & Q_{T,m+1}(x)%
\end{array}%
\right],$$
with $P_{T,n}(x)=xQ_{T,n}(x)+Q_{T,n-1}(x)$. Using the Proposition \ref{prop:1}, for $m\geq 0$, $Q_{T,m+3}(x)=x^{2}Q_{T,m+2}(x)+xQ_{T,m+1}(x)+tQ_{T,m}(x)$. Then, by induction hypothesis we obtain
\begin{align*}
Q_{S}(x)\cdot S^{m+1}(x)&=\left(Q_{S}(x)\cdot S^{m}(x)\right)\cdot S(x)\\
&=\left[ 
\begin{array}{ccc}
Q_{T,m+4}(x) & P_{T,m+3}(x) & Q_{T,m+3}(x) \\ 
Q_{T,m+3}(x) & P_{T,m+2}(x)& Q_{T,m+2}(x) \\ 
Q_{T,m+2}(x) & P_{T,m+1}(x) & Q_{T,m+1}(x)
\end{array}
\right]\left[ 
\begin{array}{ccc}
x^{2} & x & 1 \\ 
1 & 0 & 0 \\ 
0 & 1 & 0
\end{array}
\right]\\
&=\left[ 
\begin{array}{ccc}
Q_{T,m+5}(x) & P_{T,m+4}(x) & Q_{T,m+4}(x) \\ 
Q_{T,m+4}(x) & P_{T,m+3}(x) & Q_{T,m+3}(x) \\ 
Q_{T,m+3}(x) & P_{T,m+2}(x) & Q_{T,m+2}(x)
\end{array}
\right].
\end{align*}
Hence, the Eq. (\ref{eq:13}) holds for all $n\geq0$.
\end{proof}

\begin{corollary}
For $n\geq0$, 
\begin{equation}
Q_{T,n+2}(x)=Q_{T,2}(x)T_{n+1}(x)+(xQ_{T,1}(x)+Q_{T,0}(x))T_{n}(x)+Q_{T,1}(x)T_{n-1}(x),
\end{equation}
with $T_{-1}(x)=0$ for convenience.
\end{corollary}
\begin{proof}
The proof can be easily seen by the coefficient (3,1) of $Q_{S}(x)\cdot S^{n}(x)$ and the Eq. (\ref{eq:11}).
\end{proof}

\section{Conclusion}
This study examines and studied Tribonacci and Tribonacci-Lucas quaternion polynomials with the help of a simple formula. For this purpose, Tribonacci and Tribonacci-Lucas polynomials was used and examined in detail particularly in Section 1, and it was shown that this sequences to generalize the Tribonacci and Tribonacci-Lucas numbers on quaternions. In this study, Binet formulas, generating functions, matrix representation and some properties of Tribonacci and Tribonacci-Lucas quaternion polynomials were obtained. Quaternions have great importance as they are used in quantum physics, applied mathematics, graph theory and differential equations. Thus, in our future studies we plan to examine Tribonacci and Tribonacci-Lucas octonion polynomials and their key features.

\end{document}